\documentclass[11pt]{amsart}

\usepackage{amsmath, amssymb, amscd, graphicx}
\usepackage{enumerate}
\usepackage{tikz-cd}
\usepackage{subcaption}

\newtheorem{thm}{Theorem}[section]
\newtheorem{conj}[thm]{Conjecture}

\newtheorem{lem}[thm]{Lemma}

\newtheorem*{ques*}{Question}

\newtheorem*{claim*}{Claim}
\theoremstyle{remark}
\newtheorem{example}[thm]{Example}

\newtheorem*{example*}{Example}

\numberwithin{equation}{section}

 \def\wt{{\mathrm{wt}}}

\def\bR{{\mathbb R}}
\def\bZ{{\mathbb Z}}
\newcommand{\norm}[1]{\left| #1 \right|^2}

\title[Tree Lattices and the Neumann-Zagier conjecture]{Tree Lattices and the Neumann-Zagier conjecture}

\author{Joshua Evan Greene, Ahmad Issa, and Duncan McCoy}

\address{Department of Mathematics, Boston College\\ Chestnut Hill, MA 02467, USA}

\email{joshua.greene@bc.edu}

\address{Independent Researcher}

\email{ahmadissa@gmail.com}

\address{Département de mathématiques, UQAM, Montréal, QC H3C 3P8, Canada}

\email{mc\_coy.duncan@uqam.ca}

\begin{document}

\begin{abstract}
    A tree with integer weights on its vertices presents an integral lattice called a tree lattice.
    We prove that if a tree lattice is positive definite, then one of the vertex elements is irreducible, i.e.~determines a facet of its Voronoi polyhedron.
    As a corollary, we settle a 1985 conjecture of Neumann and Zagier: if a tree lattice is isometric to the lattice of integer points in Euclidean space, then the tree contains a vertex of weight one.

\end{abstract}
    
\maketitle
\section{Introduction}

A nonzero element $x$ in a positive definite lattice $L$ is {\em irreducible} if in every expression $x=y+z$ with nonzero elements $y,z\in L$, we have $y \cdot z < 0$.
Irreducible elements play an important role in the theory of positive definite lattices.
For example, they determine the facets of the Voronoi polyhedron of $L$, which consists of the points in $L \otimes \bR$ that are at least as close to the origin as to any other point of $L$.
Irreducible elements are especially important in algebraic graph theory \cite{bacheretal:lattices,GodsilRoyle2001,Greene2013}.
For example, every vertex in a graph that is not a cut-vertex determines an irreducible element of its lattice of integral cuts.

A finite, weighted tree $T$ with vertex set $V$, edge set $E$, and weight function $\wt : V \longrightarrow \bZ$ determines an integral {\em tree lattice} $L(T)$.
It is defined by equipping the free abelian group generated by $V$ with the symmetric, bilinear form whose value on a pair of generators $u,v \in V$ is given by 
\[
u \cdot v = 
\begin{cases}
    \wt(u), & u=v; \\
    -1, & (u,v) \in E; \textup{ and} \\
    0, & u \ne v, (u,v) \notin E.
\end{cases}
\]

\noindent
Motivated by a conjecture of Neumann and Zagier about tree lattices, we prove the following:

\begin{thm}
\label{thm:main}
Every positive definite tree lattice contains an irreducible vertex.    
\end{thm}

\noindent
Theorem~\ref{thm:main} follows from a sharper statement.
Recall that to every vertex $v \in V$ there corresponds a dual element $v^* \in L(T)^* \subset L(T) \otimes \bR$ defined by the property that 
$v^* \cdot v = 1$ and $v^* \cdot u = 0$ for all $u \in V$, $u \ne v$.

\begin{thm}
\label{thm:longest_dual}
    Suppose that $L(T)$ is a positive definite tree lattice.
    Let $v$ be a vertex such that $|v^*|^2\geq |u^*|^2$ for every vertex $u \in V$.
    Then $v$ is irreducible.
\end{thm}

\noindent
The inverse of the Gram matrix of the basis $V \subset L(T)$ is the Gram matrix of the dual basis $V^* \subset L(T)^*$.
Its diagonal displays the values $|v^*|^2$ for $v \in V$.
Thus, an irreducible vertex in $L(T)$ is identified by the largest diagonal entry in the inverse of the Gram matrix.

Let $I_n$ denote the lattice of integer points in Euclidean $\bR^n$.
Since every irreducible element in $I_n$ has self-pairing one, Theorem~\ref{thm:main} implies the aforementioned conjecture \cite{NeumannZagier1985,Neumann1989}:

\begin{conj} [Neumann-Zagier]
\label{conj: nz}
Let $T$ be a weighted tree such that $L(T)$ is isometric to $I_n$. Then there exists a vertex $v \in V$ with weight $\wt(v) = 1$.
\end{conj}

Neumann and Zagier were motivated by a problem in low-dimensional topology about the 3-dimensional homology cobordism group $\Theta^3_\bZ$.
Let $Y$ denote a Seifert-fibered integer homology 3-sphere with base space the 2-sphere.
There are two related obstructions to $[Y] = 0 \in \Theta^3_\bZ$ coming from gauge theory.
They both involve tree lattices, which arise as follows.
Up to orientation, $Y$ bounds a smooth plumbed 4-manifold $X$ for which the associated intersection pairing $(H_2(X),Q_X)$ is a positive definite tree lattice $L(T)$.
The tree $T$ in this case is star-shaped, meaning that it has a single central vertex $v$ of degree $\deg(v) \ge 3$.

Suppose that $[Y] = 0 \in \Theta^3_\bZ$, so $Y$ bounds a smooth integer homology 4-ball $W$.
Then $X \cup (-W)$ is a smooth, closed, positive definite 4-manifold.
Donaldson's Theorem implies that its intersection pairing is isometric $I_n$, $n=b_2(X)$ \cite{Donaldson1987}.
It follows that $L(T)$ is isometric to $I_n$, which places a stringent restriction on $Y$.
This can happen in interesting ways, as in Figure~\ref{fig:7123}.
Fintushel and Stern defined an invariant $R(Y) \in \bZ$ with the property that $R(Y) \le 0$ if $[Y] = 0 \in \Theta^3_\bZ$ \cite{FintushelStern1985}.
Neumann and Zagier showed that $R(Y) = 2 \wt(v)-3$, where $v$ denotes the central vertex in $T$ \cite{NeumannZagier1985}.
Therefore, $[Y]=0 \in \Theta^3_\bZ$ implies $\wt(v) = 1$.
Neumann and Zagier went on to conjecture that the Donaldson obstruction implies the Fintushel-Stern obstruction: if $T$ is star-shaped with central vertex $v$, if every vertex $u \ne v$ has weight $\wt(u) \ge 2$, and if $L(T)$ is isometric to $I_n$, then $\wt(v) = 1$.
This conjecture was settled by Issa and McCoy using techniques very specific to this class of trees \cite{IssaMcCoy2020}.
Sensing the phenomenon extended beyond star-shaped trees, Neumann and Zagier went on to pose Conjecture \ref{conj: nz}.

 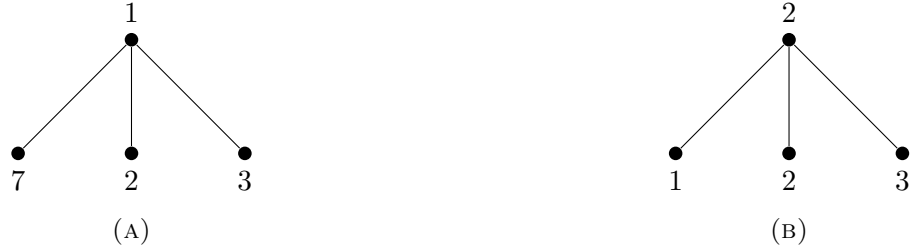
\begin{figure}
     \centering
     \begin{subfigure}{0.45\textwidth}
 \centering
 \begin{tikzpicture}[
     dot/.style={
         circle,
         fill=black,
         inner sep=1.8pt
     },
     every edge/.style={draw, thick}
 ]

 \node[dot,label=above:$1$] (c) at (0,0) {};

 \node[dot,label=below:$7$]  (l1) at (-1.5,-1.5) {};
 \node[dot,label=below:$2$] (l2) at (0,-1.5) {};
 \node[dot,label=below:$3$] (l3) at (1.5,-1.5) {};

 \draw (c) -- (l1);
 \draw (c) -- (l2);
 \draw (c) -- (l3);

 \end{tikzpicture}
 \caption{}
 \end{subfigure}
 \hfill
     \begin{subfigure}{0.45\textwidth}
     \centering
     \begin{tikzpicture}[
     dot/.style={
         circle,
         fill=black,
         inner sep=1.8pt
     },
     every edge/.style={draw, thick}
 ]

 \node[dot,label=above:$2$] (c) at (0,0) {};

 \node[dot,label=below:$1$]  (l1) at (-1.5,-1.5) {};
 \node[dot,label=below:$2$] (l2) at (0,-1.5) {};
 \node[dot,label=below:$3$] (l3) at (1.5,-1.5) {};

 \draw (c) -- (l1);
 \draw (c) -- (l2);
 \draw (c) -- (l3);

 \end{tikzpicture}
 \caption{}
 \end{subfigure}
 \caption{Two weighted trees presenting the diagonal lattice $I_4$. In both cases, there is a unique irreducible vertex and a unique bad vertex $v$, meaning $\wt(v) < \deg(v)$. In (A) the bad vertex is irreducible, and in (B) it is not.}
     \label{fig:7123}
 \end{figure}

Some examples serve to underscore the optimality of Theorem~\ref{thm:main}.

\begin{example} Theorem~\ref{thm:main} does not generalize to positive definite graph lattices.
Let $G$ be the weighted graph with two vertices $u$ and $v$, three edges between $u$ and $v$, and weights $\wt(u) = 5$ and $\wt(v) = 2$.
Then $L(G)$ is isometric to $I_2 = \langle e_1, e_2 \rangle$ by the embedding $u \longmapsto e_1 + 2e_2$, $v \longmapsto -e_1-e_2$.
However, both vertices of $G$ are reducible in $L(G)$.
\end{example}

\begin{example}  Likewise, Theorem~\ref{thm:main} does not generalize to the assertion that there exists a vertex which attains the minimum self-pairing of a nonzero element in $L(T).$
For the weighted tree $T$ of Figure~\ref{fig:e8ish}, the tree lattice $L(T)$ is isometric to $E_8 \oplus I_1$, embedding the vertices of the  $E_8$-subtree into the $E_8$-summand, and embedding the leaf of weight 31 as the negative of the dual to the central vertex in $E_8$, plus a unit vector in $I_1$.
The vertices of weight 2 are all irreducible.
There is no vertex of weight 1 in $T$, even though $L(T)$ contains a unit vector.
\end{example}
\begin{figure}
    \centering
\begin{tikzpicture}[
    dot/.style={
        circle,
        fill=black,
        inner sep=1.8pt
    },
    every edge/.style={draw, thick}
]

\node[dot,label=above:$2$] (v1) at (0,0) {};
\node[dot,label=above:$2$] (v2) at (1.5,0) {};
\node[dot,label=above:$2$] (v3) at (3,0) {};
\node[dot,label=above:$2$] (v4) at (4.5,0) {};
\node[dot,label=above:$2$] (v5) at (6,0) {};
\node[dot,label=above:$2$] (v6) at (7.5,0) {};
\node[dot,label=above:$2$] (v7) at (9,0) {};

\node[dot,label=left:$2$] (v8) at (2.25,-1.5) {};

\node[dot,label=right:$31$] (v9) at (3.75,-1.5) {};

\draw (v1) -- (v2);
\draw (v2) -- (v3);
\draw (v3) -- (v4);
\draw (v4) -- (v5);
\draw (v5) -- (v6);
\draw (v6) -- (v7);

\draw (v3) -- (v8);
\draw (v3) -- (v9);

\end{tikzpicture}
    \caption{A weighted tree presenting $E_8\oplus I_1$ which does not contain any vertices of weight one.}
    \label{fig:e8ish}
\end{figure}
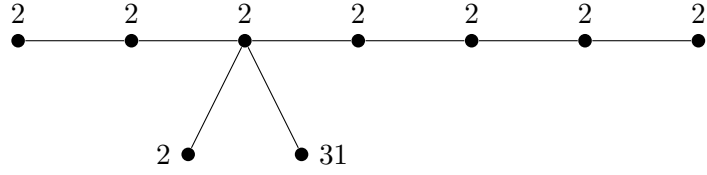
\subsection*{AI Methodology.}
This paper began as an attempt on Conjecture~\ref{conj: nz} in 2020.
At that time, JEG and DM established it under the hypothesis that $T$ has a single bad vertex, i.e.~a single vertex $v$ for which the weight $\wt(v)$ is smaller than the vertex degree $\deg(v)$.
In 2023, JEG and DM made a renewed attempt and improved the result to show that if $L(T)$ is positive definite and $T$ contains a single bad vertex, then $L(T)$ contains an irreducible vertex.
In 2026, JEG and DM submitted this result to the First Proof AI benchmark test \cite{firstproof}.
Soon thereafter, JEG and DM were contacted by AI (the second author), who had used AI (artificial intelligence) to prove Theorem~\ref{thm:main}.
We all three then set about digesting the AI argument.
The AI proof contained two main innovations: the use of the Perron-Frobenius theorem to identify a vertex, and an inductive proof akin to Lemma~\ref{lem:lengthbound} to establish its irreducibility.
While the use of Perron-Frobenius appeared to be a stroke of genius, it turned out to be a red herring, and we simplified the proof to obtain the stronger and more natural assertion of Theorem~\ref{thm:longest_dual}.
Lemma~\ref{lem:lengthbound} persisted as an essential linchpin in the argument, although the opaque inductive assertion the AI had suggested became the more transparent \eqref{eq:inductivebound}.

The initial AI proof of Theorem~\ref{thm:main} used ChatGPT 5.5 Pro and Opus 4.8. JEG used ChatGPT for Academic Researchers.
We chucked AI efforts to write this paper in favor of our own, although we engaged it for proofreading.

\subsection*{Acknowledgements}
DM is partially supported by NSERC grant RGPIN-2020-05491 and a Canada Research Chair.
JEG is partially supported by NSF Award DMS-2304856.

\section{Preparation}

We assume henceforth that $T$ is a weighted tree and $L(T)$ is positive definite.
This allows us to extend the pairing on $L(T)$ to a positive definite pairing on $L(T) \otimes \bR$ and to equate $L(T) \otimes \bR$ with its dual.
In particular, the dual $L(T)^*$ to $L(T)$ isometrically embeds in $L(T) \otimes \bR$ as the set of points with integral pairings against the elements of $L(T)$.
The lattice $L(T)$ has a distinguished basis $V$, and we let $V^*$ denote the dual basis of the dual lattice $L(T)^*$.
We use Latin letters $r$, $v$, etc. to denote vertices of $T$, and we write $r^*,v^*$, etc. to refer to their dual elements.
We also use Latin letters to refer more generally to elements of $L(T)$.
Besides the notation reserved for dual elements to vertices, we use Greek letters $\lambda, \xi, \eta$ to refer to elements of the dual lattice $L(T)^*$ and more generally $L(T) \otimes \bR$.
We write an element like $\xi \in L(T) \otimes \bR$ in terms of the basis $V$ of $L(T)\otimes \bR$ in coordinates as follows:
\[
\xi = \sum_{x \in V} \xi_x \, x.
\]
In the special case that $\xi = r^*, v^*$, we approximate this notation by using coefficients $\rho_x, \nu_x$.

\begin{lem}\label{lem:dual_positivity}
    Let $L(T)$ be a positive definite tree lattice.
    Let
    \[
    v^*=\sum_{x\in V} \nu_x x
    \]
    be the dual to a vertex $v$. Then $\nu_x>0$ for all $x \in V$.
\end{lem}

\noindent

\begin{proof}
    Define
    \[
    P=\{x\in V \, | \, \nu_x>0\}, \quad N = \{ x \in V \, | \, \nu_x < 0\},
    \]
    and define
    \[
    p= \sum_{x\in P}\nu_x x, \quad n = \sum_{x \in N} \nu_x x.
    \]
    We seek to show that $P = V$.
    By construction, $v^*=p+n$.
    Suppose for a contradiction that $N \ne \emptyset$.
    Then $n\neq 0$, and since $L(T)$ is positive definite, we have $|n|^2>0$.
    Thus, in the expansion of $|n|^2 = n \cdot \left( \sum_{x \in N} \nu_x x \right)$, there is at least one positive term $n \cdot (\nu_x x)$, $x \in N$.
    Since $\nu_x < 0$, we have $x \cdot n < 0$, and since $x \not \in P$, we have $x \cdot p\leq 0$.
    It follows that $0 \le x \cdot v^* = x \cdot (p+n)<0$, a contradiction.
    Thus $N = \emptyset$.
    Suppose for one more contradiction that there exists a vertex $y$ such that $\nu_y = 0$.
    Since $N = \emptyset$, we can select some such $y$ that is adjacent to a vertex $x \in P$.
    Then $0 \le y \cdot v^* = y \cdot p \le - \nu_x < 0$, a contradiction.
    It follows that $P = V$, as required.
\end{proof}

If $r$ is a distinguished {\em root} of $T$, then for each vertex $v \in V$, there is a unique weighted descendent subtree $T_v$ rooted at $v$: its vertex set consists of all vertices $w$ such that $v$ appears along the unique path in $T$ from $r$ to $w$.
The lattice $L(T_v)$ is a positive definite sublattice of $L(T)$, for all $v \in V$.
The dual lattice $L(T_v)^*$ sits within $L(T_v) \otimes \bR$ and in turn within $L(T) \otimes \bR$.

The inner product on $L(T) \otimes \bR$ furnishes orthogonal projections onto its subspaces $L(T_v) \otimes \bR$.
Given $\xi \in L(T) \otimes \bR$, we let $\xi^v$ denote its orthogonal projection onto $L(T_v) \otimes \bR$.
Since $\xi \cdot \eta = \xi^v \cdot \eta$ for all $\eta \in L(T_v) \otimes \bR$, we infer that $\xi^v \in L(T_v)^*$ if $\xi \in L(T)^*$.
We use the notation $\lambda^v \in L(T_v)^* \subset L(T_v) \otimes \bR \subset L(T) \otimes \bR$
for the dual element to $v$ in the basis $V(T_v)$ of $L(T_v)$.
It is the orthogonal projection of $v^*$ to $L(T_v)^*$ (the notation sidesteps the unsightly alternative $(v^*)^v$).
We stress that $\lambda^v$ is supported on $V(T_v)$, whereas $v^*$ has a nonzero coefficient on each vertex in $T$, by Lemma~\ref{lem:dual_positivity}.
This means that the elements $\lambda^v$ and $v^*$ are distinct, except in the special case that $v = r$.
Also, if $u \in V$ denotes the {\em parent} of $v$ for $v \ne r$, i.e.~the unique neighbor of $v$ not in $T_v$, then $\lambda^v$ is the orthogonal projection of $-u$ to $L(T_v) \otimes \bR$.
This observation leads to a useful expression for $\lambda^v$ in terms of $r^*$:

\begin{lem}
    \label{lem:expression}
Writing
\[
r^* =   \sum_{x \in V} \rho_x \, x,
\]
we have 
\begin{equation}
\label{eq:child_dual}
\lambda^v =\sum_{x\in V(T_v)}\frac{\rho_x}{\rho_u} \, x
\end{equation}
for every $v \in V$, where $u$ denotes the parent of $v$, if $v \ne r$, and where $\rho_u = 1$, if $v = r$. Moreover,
\begin{equation}
    \label{e:self-pairing}
|\lambda^v|^2 = \frac{\rho_v}{\rho_u}.
\end{equation}

\end{lem}

\noindent
That is, $\lambda^v$ is obtained by restricting the support of $r^*$ to $V(T_v)$ and dividing by $\rho_u$, and its self-pairing is equal to its coefficient on $v$.

\begin{proof}
Equation \eqref{eq:child_dual} is immediate if $v=r$, so assume $v \ne r$.
The right side of \eqref{eq:child_dual} is contained in $L(T_v) \otimes \bR$, since it is supported on vertices in $T_v$.
For $w \in V(T_v)$, the only pairings on the right side of
\[
0 = r^* \cdot w = \sum_{x \in V} \rho_x (x \cdot w)
\]
which are nonzero occur for $x \in V(T_v)$, if $w \ne v$, and for $x \in V(T_v) \cup \{u\}$, if $w = v$.
We infer that $\sum_{x \in V(T_v)} \rho_x x$ pairs to 0 against each $w \in V(T_v)$, $w \ne v$, and to $\rho_u$ against $v$.
These facts confirm the expression \eqref{eq:child_dual}.

We obtain \eqref{e:self-pairing} simply by pairing $\lambda^v$ against its expression  on the right side of \eqref{eq:child_dual}, invoking $\lambda^v \cdot x = \delta_{vx}$ for all $x \in V(T_v)$.
\end{proof}

For any vertex $v \in V$, the neighbors $w$ of $v$ in $T_v$ are its {\em children}.
We have an orthogonal decomposition
\[
L(T_v) \otimes \bR = \bR \lambda^v  \bigoplus_{w} L(T_w) \otimes \bR,
\]
the direct sum over all children $w$ of $v$.
We obtain a corresponding decomposition of each $\xi\in L(T_v) \otimes \bR$ as
\begin{equation}\label{eq:orthogonaldecomp}
    \xi = t \lambda^v + \sum_w \xi^w.
\end{equation}
The coefficient $t$ on $\lambda^v$ is calculated as follows.
Expressed in the basis $V(T_v)$, $t \lambda^v$ is the only term on the right side with a (possibly) nonzero coefficient on $v$.
The coefficient of $\lambda^v$ on $v$ is $\rho_v/\rho_u$, by Lemma~\ref{lem:expression}, while the coefficient of $\xi$ on $v$ is $\xi_v$, by definition.
This identifies the coefficient on $\lambda^v$ in \eqref{eq:orthogonaldecomp} as
\begin{equation}
    \label{e:t}
t = \frac{\xi_v \rho_u}{\rho_v}.
\end{equation}

\section{Proof of the main theorem}

We prove a sequence of results about irreducibility of elements in positive definite tree lattices, culminating in the proof of Theorem~\ref{thm:longest_dual}.

\begin{lem}\label{lem:lengthbound}
Let $T$ be a weighted tree such that $L(T)$ is positive definite.
Let $r$ be a vertex whose dual takes the form
\[
    r^*=\sum_{x \in V} \rho_x \, x,
\]
with $\rho_x \leq 1$ for all $x \in V$.
Then $r^*$ is irreducible in $L(T)^*$.
\end{lem}

\begin{proof}
The assertion that $r^*$ is irreducible in $L(T)^*$ is equivalent to the assertion that for any $\xi\in L(T)^*$, we have
\begin{equation}\label{eq:lengthbound}
    |\xi|^2
    \ge
    \xi \cdot r^*,
\end{equation}
with equality if and only if $\xi =0$ or $\xi=r^*$.
Indeed, simply write $r^* = \xi + (r^* - \xi)$ and examine the inner product $\xi \cdot (r^* - \xi)$.
This is what we proceed to show.

We retain the notation $u$ for the parent of $v$, if $v \ne r$, with the convention that $\rho_u = 1$, if $v = r$.
The lemma will be inferred from the following claim.
\begin{claim*}
For any vertex $v$ and for any $\xi\in L(T_v)^*$, we have the lower bound
\begin{equation}
\label{eq:inductivebound}
|\xi|^2
\ge
\rho_u|\xi_v|,
\end{equation}
with equality only if $|\xi|^2=0$ or $|\xi_v|=\frac{\rho_v}{\rho_u}$.
\end{claim*}
Specializing \eqref{eq:inductivebound} to the case $v=r$ yields \eqref{eq:lengthbound}, since then $\rho_u=1$ and $\xi_r=\xi \cdot r^*$.
Equality in \eqref{eq:inductivebound} implies that either $|\xi|^2=0$ or $|\xi|^2=|\xi_r|=\rho_r$.
If $|\xi|^2=0$, then $\xi=0$, since $L(T)$ is positive definite.
If instead $|\xi|^2= \xi \cdot r^* = \rho_r$, then
\[
|\xi- r^*|^2= |\xi|^2 -2 \xi \cdot r^* + \rho_r=0,
\]
implying that $\xi=r^*$.
Thus the claim yields \eqref{eq:lengthbound} and characterizes the cases of equality.

It remains to justify the claim, which we prove by induction on subtrees of $T$.
The base case for proving \eqref{eq:inductivebound} is the case that $T_v$ is a single vertex.
In this case, $\lambda^v=\frac{\rho_v}{\rho_u}v$.
Thus $\xi$ takes the form $\xi=m \lambda^v$ for some $m \in \bZ$, and $\xi_v=\frac{m\rho_v}{\rho_u}$. It follows that
\[
|\xi|^2= (m \lambda^v) \cdot \left(\frac{m\rho_v}{\rho_u}v\right)= \frac{m^2 \rho_v}{\rho_u}\geq |\xi_v|,
\]
with equality only if $m=0$ or $|m|=1$.
Since $0 < \rho_u\leq 1$, this gives \eqref{eq:inductivebound} in the base case.

Now we proceed with the inductive step.
Let $v$ be a vertex, and assume for each child $w$ of $v$ that the claim holds for $w$.

Pair the orthogonal decomposition of $\xi$ from \eqref{eq:orthogonaldecomp} with $v$ to obtain
\begin{equation}
\label{e:xidotv}
\xi\cdot v = t + \sum_w \xi^w\cdot v.
\end{equation}
The orthogonal decomposition \eqref{eq:orthogonaldecomp} and \eqref{e:t} also give
\begin{equation}
\label{e:xi^2}
|\xi|^2= \frac{\rho_v}{\rho_u}t^2 + \sum_w |\xi^w|^2.
\end{equation}
The projection of $v$ to $L(T_w) \otimes \bR$ is equal to $-\lambda^w$, for every child $w$ of $v$.
The pairing of $\xi^w$ with any element $\eta \in L(T) \otimes \bR$ is equal to the pairing of $\xi^w$ with $\eta^w \in L(T_w) \otimes \bR$.
Hence $\xi^w \cdot v=-\xi^w \cdot \lambda^w$ for all children $w$ of $v$.
Inserting this into \eqref{e:xidotv} and rearranging gives
\begin{equation}\label{eq:identity1}
    \sum_w \xi^w \cdot \lambda^w=t-\xi\cdot v.
\end{equation}
Now apply the inductive bound \eqref{eq:inductivebound} to each term $|\xi^w|^2$ appearing in \eqref{e:xi^2} and use the identity \eqref{eq:identity1} to get the bound
\begin{align*}
    |\xi|^2 &\geq  \frac{\rho_v}{\rho_u}t^2 + \sum_w \rho_v |\xi^w\cdot \lambda^w|\\
    &\geq \frac{\rho_v}{\rho_u}t^2 + \rho_v\left|\sum_w  \xi^w\cdot \lambda^w\right|\\
    &= \frac{\rho_v}{\rho_u}t^2 + \rho_v\left|t-\xi\cdot v\right|.
\end{align*}
Since $0 < \rho_u\leq 1$, this gives
\begin{equation}\label{eq:complicatedbound}
|\xi|^2\geq \rho_v(t^2 + \left|t-\xi\cdot v\right|).
\end{equation}
Since $\xi\in L(T_v)^*$, we have $\xi\cdot v \in \bZ$. For any $t\in \bR$ and any $m\in \bZ$, we have  $|t-m| \ge |t|(1-|t|)$, with equality only if $|t|=0$ or $|t|=1$. Thus \eqref{eq:complicatedbound} implies that  
\[
|\xi|^2\geq \rho_v|t|=\rho_u|\xi_v|,
\]
with equality only if $|t|=0$ or $|t|=1$.
This completes the inductive step.
\end{proof}

\begin{lem}\label{lem:irred_criterion}
    Let $T$ be a weighted tree with root $r$ such that $L(T)$ is positive definite.
    If $\rho_r>\rho_x$ for all $x \neq r$, then $r$ is irreducible in $L(T)$.
\end{lem}

\begin{proof}
Apply \eqref{eq:child_dual} to each neighbor $v$ of $r$.
Then $\lambda^v \in L(T_v)^*$ fulfills the hypothesis of Lemma~\ref{lem:lengthbound}.
Hence $\lambda^v \in L(T_v)^*$ is irreducible.
Moreover, since the coefficients of $\lambda^v$ are nonintegral, we see that $\lambda^v\not\in L(T_v)$.

Consider an expression $r=y+z$ with $y,z\in L(T)$ and $y\cdot z\geq 0$.
The orthogonal decomposition for $r$ as in \eqref{eq:orthogonaldecomp} takes the form
\[
r= \overline{r} -\sum_v \lambda^v,
\]
where $\overline{r} \in \bR r^*$ and the sum is over all neighbors $v$ of $r$.
The corresponding orthogonal decompositions for $y$ and $z$ take the form
\[
\text{$y=m_y \,\overline{r}+ \sum_v y^v$ \quad and \quad $z=m_z\, \overline{r}+ \sum_v z^v$,}
\]
where $m_y,m_z\in\bZ$ and $y^v,z^v\in L(T_v)^*$ obey $m_y+m_z=1$ and $y^v+z^v=-\lambda^v$ for all $v$.

Since $m_y,m_z \in \bZ$ and $m_y+m_z=1$, we have $m_y m_z\leq 0$.
The irreducibility of $\lambda^v$ in $L(T_v)^*$ implies that $y^v\cdot z^v\leq 0$ for all $v$.
As we assume that
\[
y\cdot z= m_y m_z\norm{\overline{r}} + \sum_v y^v\cdot z^v\geq 0
\]
and $\norm{\overline{r}}>0$, this implies that $m_y m_z=0$ and that $y^v\cdot z^v=0$ for all $v$.
Thus w.l.o.g.~we may assume that $m_y=0$. This implies $y=\sum_v y^v$ and hence that $y^v \in L(T_v)$ for all $v$.
However, the irreducibility of the $\lambda^v\in L(T_v)^*$ implies that for all $v$ we have $y^v=-\lambda^v$ or $y^v=0$.
As $\lambda^v \not\in L(T_v)$ for all $v$, this implies that $y^v=0$ for all $v$.
Thus we have that $y=0$ and consequently that $r$ is irreducible.
\end{proof}

\begin{proof}[Proof of Theorem~\ref{thm:longest_dual}]
    Let $v$ be a vertex such that $|v^*|^2\geq |x^*|^2$ for all vertices $x\in V$.
    Since $L(T)$ is positive definite, for any $x\neq v$, we have that
    \[
    0<|v^*-x^*|^2= |v^*|^2-2v^*\cdot x^* + |x^*|^2\leq 2|v^*|^2-2v^*\cdot x^*.
    \]
    In particular, $v^*\cdot x^*<|v^*|^2$. However, 
    \[
    v^*=\sum_{x\in V} (v^*\cdot x^*) \, x.
    \]
    Thus Lemma~\ref{lem:irred_criterion} applies to show that $v$ is irreducible.
\end{proof}

\bibliographystyle{amsalpha2}

\bibliography{references}

\end{document}